\documentclass[12pt]{amsart}
\usepackage{amssymb,amsmath,amsthm,a4wide,graphicx,epsfig,verbatim}
\usepackage{units,lineno}

\newtheorem{lemma}{Lemma}
\newtheorem{theorem}{Theorem}
\newtheorem{remark}{Remark}
\newtheorem{corollary}{Corollary}

\def\dd {{\,\rm d}}

\numberwithin{equation}{section}

\title[Dubinski{\u\i}'s Compactness Theorem]{Reflections on Dubinski{\u\i}'s\\ nonlinear compact embedding theorem}
\author{\em John W. Barrett \& Endre S\"uli}
\address{Department of Mathematics, Imperial College London, London SW7 2AZ, UK}
\address{Mathematical Institute, University of Oxford, Oxford OX1 3LB, UK}
\date{\today}

\begin{document}

\begin{abstract}
We present an overview of a result by Ju. A. Dubinski{\u\i} [Mat. Sb. 67 (109) (1965); translated in Amer. Math. Soc. Transl. (2) 67 (1968)], concerning the compact embedding of a seminormed set in $L^p(0,T; \mathcal{A}_0)$, where $\mathcal{A}_0$ is a Banach space and $p \in [1,\infty]$; we establish a variant of Dubinski{\u\i}'s theorem, where a seminormed nonnegative cone is used instead of a seminormed set; and we explore the connections of these results with a nonlinear compact embedding theorem due to E. Maitre [Int. J. Math. Math. Sci. 27 (2003)].
\end{abstract}

\maketitle


\section{Introduction}
\label{sec:intro}

Dubinski{\u\i}'s theorem concerning the compact embedding of spaces of vector-valued functions was published (in Russian) in 1965 (see \cite{DUB}), as an extension of an earlier result in this direction by Aubin, which appeared in 1963 (see \cite{aubin}).
The English translation of Dubinski{\u\i}'s original paper was included in a collection of articles by Soviet
mathematicians that was published by the American Mathematical Society in 1968 (see \cite{dubinskii-book}, pp.~226--258); the key theorem and its proof were also reproduced in an abridged form by J.-L. Lions in the, more accessible, 1969 monograph {\em Quelques m\'ethodes de r\'esolution des probl\`emes aux limites non lin\'eaire} \cite{lions-book} (cf. Sec. 12.2, and in
particular Theorem 12.1 on p.~141). Dubinski{\u\i}'s result was referenced in Simon's highly-cited 1987 article
{\em Compact Sets in the Space $L^p(0, T ; B)$} (cf. \cite{Simon}, p.~67, and Remark 8.2 on p.~86), as well as, for example, in Amann \cite{amann}, where Simon's results were further sharpened, and in the book of Roub\'{\i}\v{c}ek \cite{roubicek-book} (cf. footnote 10 on p.~194), where an extension to locally convex metrizable Hausdorff spaces is considered (see also \cite{roubicek} and \cite{kenmochi}).
A nonlinear counterpart of Simon's compactness result, which arises naturally in the study of doubly nonlinear equations of elliptic-parabolic type, was established by Maitre \cite{maitre}, whose work was motivated by the papers of Simon and Amann in the linear setting, and by a nonlinear compactness argument of Alt and Luckhaus \cite{alt-luckhaus}.

Despite its generality and usefulness in the analysis of time-dependent nonlinear partial differential equations,
Dubinski{\u\i}'s result appears to be relatively little known, --- it is certainly much less well-known than the, more familiar, Aubin--Lions--Simon compact embedding theorem.
The aim of the present paper is to review Dubinski{\u\i}'s compact embedding theorem, filling in the missing details (some of which are nontrivial) in his proof of the theorem, extending Lions' proof of the theorem to cases that were not covered in \cite{lions-book}, and correcting some minor errors in Dubinski{\u\i}'s original paper \cite{DUB} (and its
English translation \cite{dubinskii-book}). We shall also consider situations where Dubinski{\u\i}'s theorem can be deduced from Maitre's nonlinear compactness lemma \cite{maitre}.

Let $\mathcal{A}$ be a linear space over the field $\mathbb{R}$ of real numbers, and suppose that ${\mathcal{M}}_+$ is a {\em
nonnegative cone in $\mathcal{A}$}, i.e., a subset of $\mathcal{A}$ such that
\begin{equation}
\label{eq:property}
(\forall \varphi \in {\mathcal{M}}_+)\; (\forall c \in \mathbb{R}_{\geq 0})\;\;\; c\, \varphi \in {\mathcal{M}}_+.
\end{equation}
\begin{figure}\label{fig}
~\hspace{10mm}~\epsfig{file=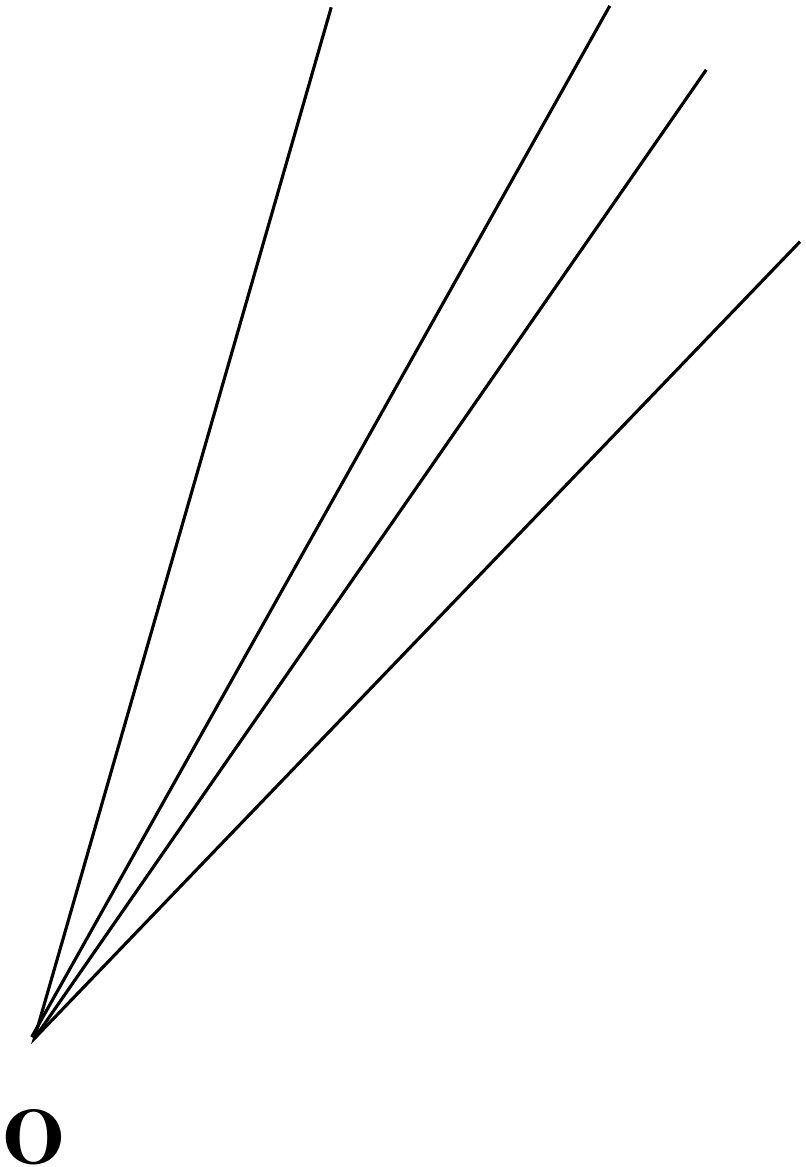, width=4cm, height=5cm}
\caption{A nonnegative cone ${\mathcal{M}}_+$ in a linear space $\mathcal{A}$.}
\end{figure}

In other words, whenever $\varphi$ is contained in ${\mathcal{M}}_+$, the ray through $\varphi$ from the origin of the linear space $\mathcal{A}$, consisting of nonnegative scalar multiples of $\varphi$ (the scalar multiplication being the one defined in the
linear space $\mathcal{A}$, which contains ${\mathcal{M}}_+$) is also contained in ${\mathcal{M}}_+$ (cf. Figure \ref{fig}).  Note in particular that while any set ${\mathcal{M}}_+$ with property \eqref{eq:property} must contain the zero element of the linear space $\mathcal{A}$, the set ${\mathcal{M}}_+$ need not be closed under summation. The linear space $\mathcal{A}$ will be referred to as the {\em ambient space} for ${\mathcal{M}}_+$.

Suppose further that each element $\varphi$ of a nonnegative cone ${\mathcal{M}}_+$ in a linear space $\mathcal A$
is assigned a certain real number, denoted by $[\varphi]_{{\mathcal{M}}_+}$, such that:
\begin{enumerate}
\item[(i)] $[\varphi]_{{\mathcal{M}}_+} \geq 0$; and $[\varphi]_{{\mathcal{M}}_+} = 0$ if, and only if, $\varphi=0$; and
\item[(ii)] $(\forall c \in \mathbb{R}_{\geq 0})\; [c\, \varphi]_{{\mathcal{M}}_+} = c\,[\varphi]_{{\mathcal{M}}_+}$.
\end{enumerate}
We shall then say that ${\mathcal{M}}_+$ is a {\em seminormed nonnegative cone}.

A subset $\mathcal{B}$ of a seminormed nonnegative cone ${\mathcal{M}}_+$ in $\mathcal{A}$ is said to be {\em bounded} if there exists a positive constant $K_0$ such that $[\varphi]_{{\mathcal{M}}_+} \leq K_0$ for all $\varphi \in \mathcal{B}$.

A seminormed nonnegative cone ${\mathcal{M}}_+$ contained in a normed linear space $\mathcal{A}$ with norm $\|\cdot\|_{\mathcal A}$ is said to be {\em embedded in $\mathcal{A}$}, and we write ${\mathcal{M}}_+ \hookrightarrow \mathcal{A}$, if
the inclusion map $i : \varphi \in {\mathcal{M}}_+ \mapsto i(\varphi)=\varphi \in \mathcal{A}$ (which is, by definition, injective and positively 1-homogeneous, i.e., $i(c\, \varphi) = c\, i(\varphi)$ for all $c \in \mathbb{R}_{\geq 0}$ and all $\varphi \in {\mathcal{M}}_+$) is a bounded operator, i.e.,
\[ (\exists K_0 \in \mathbb{R}_{>0})\; (\forall \varphi \in {\mathcal{M}}_+)\;\;\; \|i(\varphi) \|_{\mathcal A} \leq K_0 [\varphi]_{{\mathcal{M}}_+}.
\]
The symbol $i(~)$ is usually omitted from the notation $i(\varphi)$, and $\varphi \in {\mathcal{M}}_+$ is simply identified with $\varphi \in \mathcal{A}$. Thus, a bounded subset of a seminormed nonnegative cone is also a bounded subset of the ambient normed linear space the seminormed nonnegative cone is embedded in.
The embedding of a seminormed nonnegative cone ${\mathcal{M}}_+$ in a normed linear space $\mathcal{A}$ is said to be (sequentially) {\em compact} if from any infinite, bounded set of elements of ${\mathcal{M}}_+$ one can extract a subsequence that converges in $\mathcal{A}$; we shall write ${\mathcal{M}}_+ \hookrightarrow\!\!\!\rightarrow \mathcal{A}$ to denote that ${\mathcal{M}}_+$ is compactly embedded in $\mathcal{A}$. As in any metric space, and thereby also in any normed linear space,
sequential compactness and compactness are equivalent concepts, the fact that, for convenience,  we work with the notion of sequential compactness throughout this paper is of no particular significance.

Suppose that $T$ is a positive real number, $\varphi$ maps the nonempty closed interval $[0,T]$ into a seminormed nonnegative cone ${\mathcal{M}}_+$ in a normed linear space $\mathcal{A}$, and $p\in \mathbb{R}$, $p \geq 1$. Let us denote by $L^p(0,T; {\mathcal{M}}_+)$ the set of all functions $\varphi\,:\, t \in [0,T] \mapsto \varphi(t) \in {\mathcal{M}}_+$ such that
\[ \mbox{$\left(\int_0^T [\varphi(t)]^p_{{\mathcal{M}}_+} \,{\rm d}t\right)^{1/p} < \infty$.}\]
The set $L^p(0,T; {\mathcal{M}}_+)$ is then a seminormed nonnegative cone in the ambient linear space $L^p(0,T; \mathcal{A})$, with
\[\mbox{$[\varphi]_{L^p(0,T;{\mathcal{M}}_+)} := \left( \int_0^T [\varphi(t)]^p_{{\mathcal{M}}_+} \,{\rm d}t\right)^{1/p}.  $}\]
We denote by $L^\infty(0,T; {\mathcal{M}}_+)$ and $[\varphi]_{L^\infty(0,T;{\mathcal{M}}_+)}$ the usual modifications of these definitions when $p=\infty$. For a normed linear space $\mathcal{A}$, $C([0,T]; \mathcal{A})$ will denote the normed
linear space
of all continuous functions that map $[0,T]$ into $\mathcal{A}$, equipped with the norm $\|\cdot\|_{C([0,T]; \mathcal{A})}$, defined by $\|u\|_{C([0,T]; \mathcal{A})}:= \max_{t \in [0,T]}\|u\|_{\mathcal{A}}$.
For two normed linear spaces, $\mathcal{A}_0$ and $\mathcal{A}_1$,  we shall continue to denote by $\mathcal{A}_0 \hookrightarrow \mathcal{A}_1$ that $\mathcal{A}_0$ is (continuously) embedded in $\mathcal{A}_1$.

\section{Dubinski{\u\i}'s compactness theorem}
\label{sec:dubinskii}

\begin{theorem}\label{thm:Dubinski}
Suppose that $\mathcal{A}_0$ and $\mathcal{A}_1$ are Banach spaces, $\mathcal{A}_0 \hookrightarrow \mathcal{A}_1$, and that ${\mathcal{M}}_+$ is a seminormed nonnegative cone in $\mathcal{A}_0$ such that ${\mathcal{M}}_+ \hookrightarrow\!\!\!\rightarrow
\mathcal{A}_0$. Consider the set
\[ {\mathcal Y}_+:= \left\{\varphi\,:\,[0,T] \rightarrow {\mathcal{M}}_+\,:\,
[\varphi]_{L^p(0,T;{\mathcal{M}}_+)} + \left\|\frac{{\rm d}\varphi}{{\rm d}t} \right\|_{L^{p_1}(0,T;\mathcal{A}_1)}
< \infty   \right\},
\]
where $1 \leq p \leq \infty$, $1 \leq p_1 \leq \infty$, $\|\cdot\|_{\mathcal{A}_1}$ is the norm of $\mathcal{A}_1$, and ${\rm d}\varphi/{\rm d}t$ is understood in the sense of $\mathcal{A}_1$-valued distributions
on the open interval $(0,T)$.
Then, ${\mathcal Y}_+$, with
\[ [\varphi]_{{\mathcal Y}_+}:= [\varphi]_{L^p(0,T;{\mathcal{M}}_+)} + \left\|\frac{{\rm d}\varphi}{{\rm d}t} \right\|_{L^{p_1}(0,T;\mathcal{A}_1)},\]
is a seminormed nonnegative cone in the Banach space $L^p(0,T;\mathcal{A}_0)\cap W^{1,p_1}(0,T;\mathcal{A}_1)$, and
${\mathcal Y}_+ \!\hookrightarrow\!\!\!\rightarrow \! L^p(0,T; \mathcal{A}_0)$ if: {\em (a)} $1 \leq p \leq \infty$
and $1< p_1 \leq \infty$; or {\em (b)} $1\leq p < \infty$ and $p_1=1$.
\end{theorem}

Theorem \ref{thm:Dubinski} differs in two crucial aspects from the actual statement of the corresponding result in Dubinski{\u\i}'s paper (cf. Theorem 1 on p.~612 of \cite{DUB}); we now comment on these.

\begin{remark}\label{remark1}
In Theorem 1 on p.~612 of \cite{DUB}, instead of a seminormed nonnegative cone $\mathcal{M}_+$, equipped with the
functional $[~\!\cdot~\!]_{{\mathcal{M}}_+}$ satisfying hypotheses \eqref{eq:property}, (i) and (ii) above, the `light-cone' $\mathcal{M}:=\mathcal{M}_+ \cup (-\mathcal{M}_+)$ is used, equipped with a functional $[~\!\cdot~\!]_{{\mathcal{M}}}$ satisfying hypotheses \eqref{eq:property1}, (i)$'$ and (ii)$'$ in Section \ref{sec:discussion} below, which Dubinski{\u\i} calls a
seminormed set. Also, instead of our seminormed nonnegative cone ${\mathcal Y}_+$,
a seminormed set ${\mathcal Y}$ is used, which is defined by replacing the seminormed nonnegative cone $\mathcal{M}_+$
featuring in our definition of ${\mathcal Y}_+$ with the seminormed set $\mathcal{M}$. This seemingly minor alteration has
some far-reaching consequences in applications, particularly in weak-compactness arguments for sequences of nonnegative
functions (e.g. probability density functions); see, for example, \cite{BS2010-fene-M3AS,BS2010-hookean-M3AS}.
We shall elaborate on the relationship between our version of the result and the one in Dubinski{\u\i}'s paper in Section \ref{sec:discussion}, following Theorem \ref{thm:Dubinski2}.
\end{remark}

\begin{remark}\label{remark2}
Our statements of Theorem \ref{thm:Dubinski} above and Theorem \ref{thm:Dubinski2} in
Section \ref{sec:discussion} correct an oversight in Theorem 1 of Dubinski{\u\i} \cite{DUB}, which was formulated there
assuming that $1 \leq p \leq \infty$ and $1\leq p_1 \leq \infty$, failing to observe that one needs to exclude the case
$(p,p_1) = (\infty,1)$. We refer to Proposition 3 on p.~94 in Simon's paper \cite{Simon}
for a counterexample, which demonstrates that when $(p,p_1) = (\infty,1)$ the compact embedding claimed to hold in Theorem 1 of Dubinski{\u\i} \cite{DUB} is false, in the special case when $\mathcal{M}$ is a Banach space that is compactly embedded in $\mathcal{A}_0$; the same comment applies to Lemma 2 in \cite{DUB}.
In our version of Dubinski{\u\i}'s theorem this particular combination
of $p$ and $p_1$ is inadmissible, which is in agreement with Corollary 4 on p.~85 of Simon \cite{Simon} when $\mathcal M$ is a Banach space;
the subsequent comments on page p.~86 of \cite{Simon} concerning the sharpness of the hypotheses of Corollary 4 (see also Corollary 5 and the subsequent comments on p.~86) imply that
hypotheses {\textit{(a)}} and {\textit{(b)}} in Theorems \ref{thm:Dubinski} and \ref{thm:Dubinski2} here
as well as in our version of Dubinski's Lemma 2 (cf. Lemma \ref{lemma2} below) are sharp.
Lions' presentation of Dubinski{\u\i}'s theorem is restricted to the ranges $1<p<\infty$ and $1<p_1<\infty$, so the critical combination $(p,p_1)=(\infty,1)$ is {\em a priori} excluded; in particular, Theorem 12.1 on p.~141 of
\cite{lions-book} is correct as stated.
\end{remark}

The proof of Theorem \ref{thm:Dubinski} is based on some preliminary results, the first of which is the following lemma (cf.
Lemma 1 on p.~612 in \cite{DUB} and Lemma 12.1 on p.~141 in \cite{lions-book}). It can be viewed as a nonlinear version
of Ehrling's lemma \cite{Ehrling} (cf. Lions \& Magenes \cite{Lions-Magenes}, Chapter I, Theorem 16.4; or Temam \cite{Temam}, Chapter III, Lemma 2.1).

\begin{lemma}\label{lemma1}
Suppose that $\mathcal{A}_0$ and $\mathcal{A}_1$ are normed linear spaces,
$\mathcal{A}_0 \hookrightarrow \mathcal{A}_1$, and ${\mathcal{M}}_+$ is a seminormed nonnegative cone in $\mathcal{A}_0$ such that ${\mathcal{M}}_+ \hookrightarrow\!\!\!\rightarrow \mathcal{A}_0$. Then, for each
$\eta>0$ there exists a constant $c_\eta>0$ such that
\[ \| u - v \|_{{\mathcal A}_0} \leq \eta \left( [u]_{{\mathcal{M}}_+} + [v]_{{\mathcal{M}}_+}\right) + c_\eta \|u - v\|_{{\mathcal A}_1}\qquad
\forall u, v \in {\mathcal{M}}_+.\]
\end{lemma}

\begin{proof}
The proof proceeds by contradiction. Suppose for contradiction that the claim of the lemma is false;
then, there exists a real number $\eta_0>0$ such that for each $n
\in \mathbb{N}$ there exist $u_n, v_n \in {\mathcal{M}}_+$ such that
\begin{equation}\label{le1-eq1}
\| u_n - v_n \|_{{\mathcal A}_0} > \eta_0 \left( [u_n]_{{\mathcal{M}}_+} + [v_n]_{{\mathcal{M}}_+}\right) + n \|u_n - v_n\|_{{\mathcal A}_1},\qquad n=1,2,\dots.
\end{equation}
With such $u_n, v_n \in {\mathcal{M}}_+$, we must have $[u_n]_{{\mathcal{M}}_+} + [v_n]_{{\mathcal{M}}_+}>0$ for all $n \in \mathbb{N}$; else, it would follow that
both $[u_n]_{{\mathcal{M}}_+}=0$ and  $[v_n]_{{\mathcal{M}}_+}=0$ for a certain $n \in \mathbb{N}$, which, by axiom (i) of a seminormed nonnegative cone, would imply that $u_n=0$ and $v_n=0$ for some $n \in \mathbb{N}$, and therefore both sides of
\eqref{le1-eq1} would be equal to $0$ for such an $n$, which would, in turn, contradict the strict inequality relating the two sides of \eqref{le1-eq1}.

Let us now define
\begin{equation}\label{le1-eq2}
\tilde{u}_n:= \frac{u_n}{ [u_n]_{{\mathcal{M}}_+} + [v_n]_{{\mathcal{M}}_+}},\qquad   \tilde{v}_n:= \frac{v_n}{ [u_n]_{{\mathcal{M}}_+} + [v_n]_{{\mathcal{M}}_+}},\qquad n=1,2,\dots.
\end{equation}
With these definitions of $\tilde{u}_n$ and $\tilde{v}_n$ and using the positive $1$-homogeneity of the norms $\|\cdot\|_{{\mathcal A}_0}$ and $\|\cdot\|_{{\mathcal A}_1}$, we have from \eqref{le1-eq1} that
\[
\| \tilde{u}_n - \tilde{v}_n \|_{{\mathcal A}_0} > \eta_0 + n \|\tilde{u}_n - \tilde{v}_n\|_{{\mathcal A}_1},\qquad
n=1,2,\dots,
\]
and therefore
\begin{equation}\label{le1-eq3}
\eta_0 < \| \tilde{u}_n - \tilde{v}_n \|_{{\mathcal A}_0}\quad \mbox{and}\quad\|\tilde{u}_n - \tilde{v}_n\|_{{\mathcal A}_1} < \frac{1}{n} \| \tilde{u}_n - \tilde{v}_n \|_{{\mathcal A}_0},\qquad
n=1,2,\dots.
\end{equation}
Thanks to axiom (ii) of a seminormed nonnegative cone we have from \eqref{le1-eq2} that
\begin{equation}\label{le1-eq4}
[\tilde{u}_n]_{{\mathcal{M}}_+} \leq 1\qquad \mbox{and}\qquad [\tilde{v}_n]_{{\mathcal{M}}_+}\leq 1,\qquad n=1,2,\dots.
\end{equation}
Since, by hypothesis, ${\mathcal{M}}_+ \hookrightarrow\!\!\!\rightarrow \mathcal{A}_0$, it follows from \eqref{le1-eq4} that there is a subsequence $\{{\tilde u}_{n_k}\}_{k=1}^\infty$ such that ${\tilde u}_{n_{k}} \rightarrow \tilde u$ in $\mathcal{A}_0$, and (noting that $[\tilde{v}_{n_k}]_{{\mathcal{M}}_+}\leq 1$ for all
$k=1,2,\dots$,) $\{{\tilde v}_{n_{k_\ell}}\}_{\ell=1}^\infty$ such that  and ${\tilde v}_{n_{k_\ell}} \rightarrow \tilde v$ in $\mathcal{A}_0$. Since all subsequences of $\{{\tilde u}_{n_k}\}_{k=1}^\infty$ converge to $\tilde u$ in $\mathcal{A}_0$, it
follows in particular that $\tilde{u}_{n_{k_\ell}} \rightarrow \tilde u$ in $\mathcal{A}_0$.
Hence, $\tilde{u}_{n_{k_\ell}} - \tilde{v}_{n_{k_\ell}} \rightarrow \tilde u- \tilde v$ in $\mathcal{A}_0$. A convergent sequence in $\mathcal{A}_0$ is bounded, and therefore there exists a positive constant $c_0$ such that $\| \tilde{u}_{n_{k_\ell}} - \tilde{v}_{n_{k_\ell}} \|_{{\mathcal A}_0} \leq c_0$ for all $\ell \geq 1$. We then deduce from the second inequality in \eqref{le1-eq3} that
\begin{equation}\label{le1-eq5}
\|\tilde{u}_{n_{k_\ell}} - \tilde{v}_{n_{k_\ell}}\|_{{\mathcal A}_1}\leq \frac{c_0}{n_{k_\ell}},\qquad \ell=1,2,\dots.
\end{equation}
Letting $\ell \rightarrow \infty$, whereby $n_{k_\ell} \rightarrow \infty$, we have from \eqref{le1-eq5}
that $\tilde{u}_{n_{k_\ell}} - \tilde{v}_{n_{k_\ell}} \rightarrow 0$
in $\mathcal{A}_1$. As $\tilde{u}_{n_{k_\ell}} - \tilde{v}_{n_{k_\ell}} \rightarrow \tilde u- \tilde v$ in $\mathcal{A}_0 (\hookrightarrow \mathcal{A}_1)$, by the uniqueness of the limit in $\mathcal{A}_1$ we then have that $\tilde{u}- \tilde{v} = 0$; i.e., $\tilde{u} = \tilde{v} \in \mathcal{A}_0$; consequently, $\tilde{u}_{n_{k_\ell}} - \tilde{v}_{n_{k_\ell}} \rightarrow \tilde{u}-\tilde{v} = 0$
in $\mathcal{A}_0$. This, however contradicts the first inequality in \eqref{le1-eq3}. That completes the proof.
\end{proof}

The second preliminary result is encapsulated in the following lemma (cf. Lemma 2 on p.~613 in \cite{DUB}).

\begin{lemma}\label{lemma2}
Suppose that ${\mathcal{M}}_+$ is a seminormed nonnegative cone in an ambient Banach space $\mathcal{A}_1$ such that
${\mathcal{M}}_+ \hookrightarrow\!\!\!\rightarrow \mathcal{A}_1$. Let $1\leq p \leq \infty$ and $1\leq  p_1 \leq \infty$,
and consider the set
\[ {\mathcal Y}_+:= \left\{\varphi\,:\,[0,T] \rightarrow {\mathcal{M}}_+\,:\,
[\varphi]_{L^p(0,T;{\mathcal{M}}_+)} + \left\|\frac{{\rm d}\varphi}{{\rm d}t} \right\|_{L^{p_1}(0,T;\mathcal{A}_1)}
< \infty   \right\}.
\]
Then, $\mathcal{Y}_+$ is a seminormed nonnegative cone in $C([0,T]; \mathcal{A}_1) \subset L^p(0,T;\mathcal{A}_1)$. Moreover,
\begin{itemize}
\item[(a)] if $1\leq p \leq \infty$ and $1< p_1 \leq \infty$, then ${\mathcal Y}_+ \hookrightarrow\!\!\!\rightarrow C([0,T]; \mathcal{A}_1)$;
\item[(b)] if $1\leq p < \infty$ and $p_1 = 1$, then ${\mathcal Y}_+ \hookrightarrow\!\!\!\rightarrow L^p(0,T;\mathcal{A}_1)$.
\end{itemize}
\end{lemma}

\begin{proof}
We shall suppose that $p \in [1,\infty)$ and $p_1 \in [1,\infty)$; the proof is easily adapted to the cases when
$p=\infty$ or $p_1=\infty$.

Let us first show that ${\mathcal Y}_+ \subset C([0,T]; \mathcal{A}_1)$. Suppose that $u \in {\mathcal Y}_+$; then,
\[\int_0^T [u(t)]_{{\mathcal{M}}_+}^p \dd t < \infty\quad \mbox{and}\quad
\int_0^T \left\|\frac{\dd u}{\dd t}(t) \right\|_{{\mathcal A}_1}^{p_1} \dd t < \infty.\]
Thus, by H\"older's inequality (and trivially if $p=1$),
\[\int_0^T [u(t)]_{{\mathcal{M}}_+} \dd t < \infty\quad \mbox{and}\quad
\int_0^T \left\|\frac{\dd u}{\dd t}(t) \right\|_{{\mathcal A}_1} \dd t < \infty.\]
Using in the first of these inequalities that ${\mathcal{M}}_+$ is (continuously) embedded in $\mathcal{A}_1$, we deduce that
\begin{equation}\label{le2-eq1a}
\int_0^T \|u(t)\|_{{\mathcal A}_1} \dd t < \infty\quad \mbox{and}\quad
\int_0^T \left\|\frac{\dd u}{\dd t}(t) \right\|_{{\mathcal A}_1} \dd t < \infty.
\end{equation}
This, in turn, implies that ${\mathcal Y}_+ \subset W^{1,1}(0,T; \mathcal{A}_1)$.
Hence, thanks to Theorem 2 (i) on p.~286 of Evans \cite{Evans}, $u$ is almost everywhere on $[0,T]$ equal to a continuous function from $[0,T]$ to $\mathcal{A}_1$, which we shall henceforth identify with $u$. Thus we have shown that ${\mathcal Y}_+ \subset C([0,T]; \mathcal{A}_1)$.
The continuous embedding of $\mathcal{Y}_+$ in $C([0,T];\mathcal{A}_1)$ then follows from Theorem 2 (iii) on p.~286 of Evans \cite{Evans}.

Next we shall show that ${\mathcal Y}_+ \hookrightarrow\!\!\!\rightarrow C([0,T]; \mathcal{A}_1)$
if $p_1>1$, and
${\mathcal Y}_+ \hookrightarrow\!\!\!\rightarrow L^p(0,T; \mathcal{A}_1)$
if $p_1=1$ and $p < \infty$. Let us suppose to this
end that $\{u_n\}_{n=1}^\infty$ is a bounded sequence in ${\mathcal Y}_+$. Thanks to the (continuous) embedding of ${\mathcal Y}_+$
in $C([0,T]; \mathcal{A}_1)$, the sequence $\{u_n\}_{n=1}^\infty$ is then also bounded in $C([0,T]; \mathcal{A}_1)$;
i.e., there exists a $K>0$ such that, for all $n \in \mathbb{N}$,
\begin{equation}\label{le2-eq2}
\int_0^T [u_n(t)]_{{\mathcal{M}}_+}^p \dd t \leq K,\quad \int_0^T \left\|\frac{\dd u_n}{\dd t}(t)\right\|^{p_1}_{{\mathcal A}_1}\dd t \leq K \quad\mbox{and}\quad \max_{t \in [0,T]}\|u_n(t)\|_{{\mathcal A}_1}
\leq K.
\end{equation}

Let us first observe that:
\begin{eqnarray}\label{le2-eq3}
&&\mbox{for every measurable set $F\subset(0,T)$ of positive Lebesgue measure}\nonumber\\
&&\mbox{there exists a $t \in F$ and a positive real number $K(t)$, such that}\\
&&\mbox{for all $n \in \mathbb{N}$ we have that $[u_n(t)]_{{\mathcal{M}}_+} \leq K(t)$}.\nonumber
\end{eqnarray}
Indeed, suppose that \eqref{le2-eq3} is false; then, there exists a set $F \subset (0,T)$ of positive Lebesgue measure
such that, for each $t \in F$ and each positive (real number, which, without loss of generality we can take to be an) integer $k$, there exists a positive integer $n_k:=n(t,k)\in \mathbb{N}$ such that $[u_{n_k}(t)]_{{\mathcal{M}}_+} > k$. Equivalently, if \eqref{le2-eq3} is false, then there exists a set
$F \subset (0,T)$ of positive Lebesgue measure
such that, for each $t \in F$, $\lim_{k \rightarrow \infty} [u_{n_k}(t)]_{{\mathcal{M}}_+} = +\infty$. As
\[ \int_0^T [u_{n_k}(t)]_{{\mathcal{M}}_+}^p \dd t \geq \int_F [u_{n_k}(t)]^p_{{\mathcal{M}}_+} \dd t\]
and the right-hand side tends to $+\infty$ as $k \rightarrow \infty$, the left-hand side must also tend to $+\infty$. The
latter however contradicts the first inequality in \eqref{le2-eq2}, which then implies that \eqref{le2-eq3} holds.
We shall use \eqref{le2-eq3} to establish the existence of a countable dense subset $G$ of the interval $(0,T)$ and
of an infinite subsequence $\{u_n\}_{n \in \mathcal{F}}$ of the sequence $\{u_n\}_{n \in \mathbb{N}}$, where
$\mathcal{F}$ is an infinite subset of $\mathbb{N}$, such that $\{u_n(t')\}_{n \in \mathcal{F}}$ converges in $\mathcal{A}_1$
for each $t' \in G$.

Let $\mathbb{Q}_{(0,1)}$ denote the set of all rational numbers contained in the open interval $(0,1)$, and define
$\mathbb{Q}_T: = T\, \mathbb{Q}_{(0,1)} = \{ T q \, : \, q \in \mathbb{Q}_{(0,1)}\}$. As $\mathbb{Q}_{(0,1)}$ is a countable
dense subset of $(0,1)$, it follows that $\mathbb{Q}_T$ is a countable dense subset of the open interval $(0,T)$.

Consider, for each $\tau \in \mathbb{Q}_T$, the infinite sequence of disjoint open subintervals of $(0,T)$:
\[ I_1(\tau) := \left(0, {\textstyle \frac{1}{2}}\tau\right),\qquad I_k(\tau) := \left(\tau \sum_{\ell = 1}^{k-1} 2^{-\ell},
\tau \sum_{\ell=1}^k 2^{-\ell}\right),\quad k = 2, 3, \dots .\]
Observe that $\mbox{meas}(I_k(\tau)) = \tau\, 2^{-k}>0$, $k=1,2, \dots$; furthermore, each element of $I_{k-1}(\tau)$ is strictly smaller than any element of $I_k(\tau)$, which is, in turn, strictly smaller than $\tau$, for all $k=1,2,\dots$; in addition, the closure of $\cup_{k=1}^\infty I_k(\tau)$ in $\mathbb{R}$ is equal to $[0,\tau]$, for all $\tau \in \mathbb{Q}_T$.

Thanks to \eqref{le2-eq3}, for each $\tau \in \mathbb{Q}_T$ and each $k=1,2,\dots$, there exists a $t_k(\tau) \in I_k(\tau)$
and a positive real number $K(t_k(\tau))$ such that
\begin{equation}\label{le2-eq4}
 [u_n(t_k(\tau))]_{{\mathcal{M}}_+} \leq K(t_k(\tau))\qquad \forall n \in \mathbb{N}.
\end{equation}
Clearly, the sequence $\{t_k(\tau)\}_{k=1}^\infty$ converges to $\tau$. Indeed, since $t_k(\tau) \in I_k(\tau)$, it follows
that $\tau \sum_{\ell = 1}^{k-1} 2^{-\ell}< t_k(\tau) < \tau$; further, for each $\varepsilon>0$ there exists a $k_0=k_0(\varepsilon,\tau)\in \mathbb{N}$ such that
\[ 1 - \sum_{\ell=1}^{k_0-1} 2^{-\ell} = \sum_{\ell=k_0}^\infty 2^{-\ell} < \frac{1}{\tau}\varepsilon.\]
Hence, also
\[\tau - \tau \sum_{\ell=1}^{k_0-1} 2^{-\ell} < \varepsilon,\]
whereby
\[0 < \tau - t_k(\tau) < \tau - \tau \sum_{\ell=1}^{k-1} 2^{-\ell} \leq \tau - \tau \sum_{\ell=1}^{k_0-1} 2^{-\ell} < \varepsilon,\]
for all $k \geq k_0(\varepsilon, \tau)$.

We define
\[ G:= \bigcup_{\tau \in \mathbb{Q}_T} \{t_1(\tau), t_2(\tau), \dots\}.\]
Since $G$ is a countable union of countable sets, it is a countable set. Furthermore, $G$ is dense in $(0,T)$. Indeed, given
any $t \in (0,T)$ and any $\varepsilon>0$, thanks to the denseness of $\mathbb{Q}_T$ in $(0,T)$ there exists a $\tau=\tau(t)
\in \mathbb{Q}_T$ such that $|t - \tau|< \varepsilon$; also, since $\lim_{k \rightarrow \infty} t_k(\tau) = \tau$, there
exists a positive integer $k_0 = k_0(\varepsilon, \tau(t))$ (as above) such that $|\tau - t_k(\tau(t))|<\varepsilon$ for all $k \geq k_0(\varepsilon, \tau(t))$;  consequently, for each $t \in (0,T)$ and each $\varepsilon>0$ there exists a $k_0=k_0(\varepsilon,\tau(t))\in \mathbb{N}$ such that $| t - t_k(\tau(t))| < 2 \varepsilon$ for all $k \geq k_0(\varepsilon,\tau(t))$, with $t_k(\tau(t)) \in G$.

Thus we have shown that $G$ is a countable dense subset of the interval $(0,T)$. As $G$ is countable, it is a
bijective image of the set $\mathbb{N}$ of all positive integers. We can therefore reindex the elements of the set $G$
with the positive integers $1, 2,\dots,$ as $t_1, t_2, \dots$; viz.
\[ G=\{t_1,t_2,\dots\}.\]
Hence we can restate \eqref{le2-eq4} equivalently as follows: for each $t_k \in G$ there exists
a positive real number $K(t_k)$, such that
\begin{equation}\label{le2-eq5}
[u_n(t_k)]_{{\mathcal{M}}_+} \leq K(t_k)\qquad \forall n \in \mathbb{N},\qquad k=1,2,\dots.
\end{equation}

Now take $k=1$ in \eqref{le2-eq5};
thanks to the assumed compact embedding of $\mathcal{M}_{+}$ in $\mathcal{A}_1$, there exists a subsequence
of $\{u_n\}_{n \in \mathbb{N}}$, which we denote by $\{u_{n}\}_{n \in \mathbb{N}(\{t_1\})}$, where $\mathbb{N}(\{t_1\})$ is
an infinite subset of $\mathbb{N}$, such that $\{u_{n}(t_1)\}_{n \in \mathbb{N}(\{t_1\})}$ converges in $\mathcal{A}_1$.

Next we take $k=2$ and $n \in \mathbb{N}(\{t_1\})$ in \eqref{le2-eq5} to deduce that
\[ [u_n(t_2)]_{\mathcal{M}_{+}} \leq K(t_2) \qquad \forall n \in \mathbb{N}(\{t_1\}).\]
Hence, by the compact embedding of $\mathcal{M}_{+}$ in $\mathcal{A}_1$, there exists a subsequence of the sequence
$\{u_n\}_{n\in\mathbb{N}(\{t_1\})}$, which we denote by $\{u_n\}_{n \in \mathbb{N}(\{t_1,t_2\})}$,
where $\mathbb{N}(\{t_1,t_2\})$ is an infinite subset of $\mathbb{N}(\{t_1\}) \subset \mathbb{N}$, such that
$\{u_n(t_2)\}_{n \in \mathbb{N}(\{t_1,t_2\})}$ converges in $\mathcal{A}_1$. Further, since
$\{u_n\}_{n\in \mathbb{N}(\{t_1,t_2\})}$ is a subsequence of $\{u_n\}_{n\in \mathbb{N}(\{t_1\})}$, and all subsequences
of a convergent sequence converge (to the same limit), we have that $\{u_n(t_1)\}_{n\in\mathbb{N}(\{t_1,t_2\})}$ converges in $\mathcal{A}_1$.

Continuing this process, we thus construct a subsequence of $\{u_n\}_{n=1}^\infty$, which we denote by $\{u_n\}_{n \in \mathbb{N}(\{t_1,t_2,\dots,t_m\})}$, such that $\mathbb{N}(\{t_1,t_2,\dots,t_m\}) \subset \cdots \subset \mathbb{N}(\{t_1,t_2\})
\subset \mathbb{N}(\{t_1\}) \subset \mathbb{N}$, and $\{u_n(t_k)\}_{n \in \mathbb{N}(\{t_1,t_2,\dots,t_m\})}$ converges in
$\mathcal{A}_1$ for all $k=1,2,\dots,m$ and all $m \geq 2$.

Repetition of this process {\em ad infinitum} results in a subsequence of $\{u_n\}_{n=1}^\infty$, indexed by integers in
an infinite index set $\mathcal{F}:=\mathbb{N}(\{t_1,t_2,\dots\}) = \mathbb{N}(G)$, which we denote by
$\{u_n\}_{n \in \mathcal{F}}$, such that
\[\mathcal{F} = \mathbb{N}(G) \subset \dots \subset \mathbb{N}(\{t_1,t_2,\dots,t_m\}) \subset \cdots \subset \mathbb{N}(\{t_1,t_2\}) \subset \mathbb{N}(\{t_1\}) \subset \mathbb{N},\]
and $\{u_n(t_k)\}_{n \in \mathcal{F}}$ converges in $\mathcal{A}_1$ for all $k=1,2,\dots$.
In other words, for each $t' \in G$, where $G$ is a countable dense subset of $[0,T]$, we have that
$\{u_{n}(t')\}_{n \in \mathcal{F}}$ is a convergent (and therefore a Cauchy) sequence in $\mathcal{A}_1$.

We shall now show that $\{u_{n}\}_{n \in \mathcal{F}}$ is a Cauchy sequence in $C([0,T]; \mathcal{A}_1)$ if $p_1>1$,
and in $L^p(0,T; \mathcal{A}_1)$ if $p_1=1$ and $p<\infty$.
Since
$\mathcal{A}_1$ is a Banach space, and therefore $C([0,T]; \mathcal{A}_1)$ is also Banach, this will then imply that
$\{u_{n}\}_{n \in \mathcal{F}}$ is convergent in $C([0,T]; \mathcal{A}_1)$ if $p_1>1$.
Thus we will have shown that
${\mathcal Y}_+ \hookrightarrow\!\!\!\rightarrow C([0,T]; \mathcal{A}_1)$
if $p_1>1$.
An analogous argument with $C([0,T]; \mathcal{A}_1)$ replaced by $L^p(0,T; \mathcal{A}_1)$
in the case $p_1=1$ and $p < \infty$
will complete the proof of the
Lemma.

We proceed as follows.
By the triangle inequality, for any $t,t' \in [0,T]$ and any $m, n \in \mathbb{N}$, we have that
\begin{eqnarray}\label{le2-triang}
\|u_m(t) - u_n(t)\|_{\mathcal{A}_1} &\leq&
\|u_m(t) - u_m(t')\|_{\mathcal{A}_1} + \|u_m(t') - u_n(t')\|_{\mathcal{A}_1}\nonumber\\
&& + \|u_n(t') - u_n(t)\|_{\mathcal{A}_1}.
\end{eqnarray}

\smallskip

\textit{(a)} Now suppose that $1 \leq p \leq \infty$ and $1<p_1 \leq \infty$.
For $p_1>1$, let $p_1':=p_1/(p_1-1)$ denote the conjugate of $p_1$. By H\"older's inequality and
the second bound in \eqref{le2-eq2}, for any $t,t' \in [0,T]$ and any $n \in \mathbb{N}$, we have that
\begin{eqnarray}\label{le2-holder}
\|u_{n}(t) - u_{n}(t')\|_{\mathcal{A}_1} &\leq& \left| \int_{t'}^t \left\|\frac{\dd u_{n}}{\dd s}(s)
\right\|_{\mathcal{A}_1} \dd s\right| \leq  \left|\int_{t'}^t \left\|\frac{\dd u_{n}}{\dd s}(s)
\right\|^{p_1}_{\mathcal{A}_1} \dd s \right|^{\frac{1}{p_1}}|t-t'|^{\frac{1}{p_1'}} \nonumber\\
&\leq& K^{\frac{1}{p_1}}|t-t'|^{\frac{1}{p_1'}}.
\end{eqnarray}
Thanks to the denseness of $G$ in $(0,T)$, for $\varepsilon>0$ fixed, and therefore
\[ \delta:= \left({\textstyle\frac{1}{3}}\varepsilon\right)^{p_1'} K^{-\frac{p_1'}{p_1}}\]
fixed, the interval $[0,T]$ can be covered by a {\em finite} number of open intervals of the form $(t'-\delta, t'+\delta)$ centred at (a finite number of) points $t' \in G$. Let us denote the corresponding finite set of such `center points' $t'$ by $G_\varepsilon$ $(\subset G)$.
Thus for any $\varepsilon>0$ and any $t \in [0,T]$ there exists a $t' \in G_\varepsilon$ such that
$|t - t'|< \delta$, or, equivalently (thanks to the definition of $\delta$), $K^{1/p_1}|t-t'|^{1/p_1'}< \frac{1}{3}\varepsilon$.

It thus follows from \eqref{le2-holder} that for any $\varepsilon>0$ and
any $t \in [0,T]$ there exists a $t' \in G_\varepsilon$ such that
\begin{equation}\label{le2-third}
\|u_{n}(t) - u_{n}(t')\|_{\mathcal{A}_1} < \textstyle{\frac{1}{3}}\varepsilon\qquad \forall n \in \mathbb{N}\quad (\mbox{and, in particular, $\forall n \in \mathcal{F}$}).
\end{equation}
Substituting \eqref{le2-third} into \eqref{le2-triang} we deduce that for any $\varepsilon>0$ and any $t \in [0,T]$ there
exists a $t' =t'(t,\varepsilon)\in G_\varepsilon$ such that
\begin{eqnarray*}
\|u_m(t) - u_n(t)\|_{\mathcal{A}_1} &<& \textstyle{\frac{2}{3}}\varepsilon + \|u_m(t') - u_n(t')\|_{\mathcal{A}_1}\nonumber\\
&&\qquad\quad \forall m, n \in \mathbb{N}\quad (\mbox{and, in particular, $\forall m, n \in \mathcal{F}$}).\nonumber
\end{eqnarray*}
Hence, for any $\varepsilon>0$, we have that
\begin{eqnarray}\label{le2-last}
\max_{t \in [0,T]}\|u_m(t) - u_n(t)\|_{\mathcal{A}_1} & < & {\textstyle{\frac{2}{3}}}\varepsilon +
\max_{t' \in G_\varepsilon} \|u_m(t') - u_n(t')\|_{\mathcal{A}_1}\nonumber\\
&&\qquad\quad \forall m, n \in \mathbb{N}\quad (\mbox{and, in particular, $\forall m, n \in \mathcal{F}$}).
\end{eqnarray}
As $\{u_n(t')\}_{n \in \mathcal{F}}$ is a Cauchy sequence in $\mathcal{A}_1$ for each $t' \in G$, and $G_\varepsilon$
is a finite set, it follows that the sequence $\{u_n(t')\}_{n \in \mathcal{F}}$ is Cauchy in $\mathcal{A}_1$ {\em uniformly} in $t' \in G_\varepsilon$; i.e., for any $\varepsilon>0$ there exists a positive integer $n_0=n_0(\varepsilon)\in \mathcal{F}$ such that
\[ \max_{t' \in G_\varepsilon}\|u_m(t') - u_n(t')\|_{\mathcal{A}_1} < {\textstyle{\frac{1}{3}}}\varepsilon\qquad
\mbox{$\forall m, n \in \mathcal{F}$ such that $m, n \geq n_0$.}\]
Substituting this into \eqref{le2-last} we deduce that for any $\varepsilon>0$ there exists a positive integer
$n_0=n_0(\varepsilon)\in \mathcal{F}$ such that
\[ \|u_m - u_n\|_{C([0,T]; \mathcal{A}_1)}:=\max_{t \in [0,T]}\|u_m(t) - u_n(t)\|_{\mathcal{A}_1} < \varepsilon \qquad
\mbox{$\forall m, n \in \mathcal{F}$ such that $m, n \geq n_0$.}\]
Thus we have shown that the sequence $\{u_n\}_{n \in \mathcal{F}}$ is Cauchy (and therefore convergent) in the Banach
space $C([0,T]; \mathcal{A}_1)$. That completes the proof of part \textit{(a)} of the Lemma.

\smallskip
\textit {(b)} Now suppose that $1 \leq p < \infty$ and $p_1=1$.
We begin by noting that,
similarly as in \eqref{le2-holder}, we have that, for any $t,t' \in [0,T]$,
\begin{eqnarray}\label{le2-holder-b}
\|u_{n}(t) - u_{n}(t')\|_{\mathcal{A}_1} &\leq& \left| \int_{t'}^t \left\|\frac{\dd u_{n}}{\dd s}(s)\right\|_{\mathcal A_1} \dd s\right|\qquad \forall n \in \mathbb{N}.
\end{eqnarray}
For $\varepsilon>0$ and $K$ as in \eqref{le2-eq2}, we take $N=N(\varepsilon) \in \mathbb{N}$ defined by
\begin{equation}\label{le32-defN}
N(\varepsilon):= \left[\left(\frac{3K}{\varepsilon}\right)^{\!p}T\right] + 1
>\left(\frac{3K}{\varepsilon}\right)^{\!p}T,
\end{equation}
where, for a real number $x>0$, $[x]$ signifies the integer part of $x$,
and subdivide the interval $[0,T]$ into $N$ subintervals
\[ [0,t_1],\;\, [t_1,t_2],\;\,\dots,\;\, [t_{N-1}, t_N],\]
each of length $h:=T/N$, where $t_k := k h$, $k=0,\dots,N$. Let us choose, in each of
the {\em open} subintervals $(t_{k-1},t_k)$ a single element $t_k' \in G$, $k=1,\dots, N$, and
define
\[G_\varepsilon:=\{t_1', t_2', \dots, t_N'\}.\]
It follows from \eqref{le2-holder-b}, with $t \in [t_{k-1},t_k]$ and $t'=t'_k$, that
\begin{eqnarray*}
\int_{t_{k-1}}^{t_k}\|u_{n}(t) - u_{n}(t'_k)\|^p_{\mathcal{A}_1} \dd t &\leq& \int_{t_{k-1}}^{t_k} \left| \int_{t'_k}^t \left\|\frac{\dd u_{n}}{\dd s}(s)\right\|_{\mathcal A_1} \dd s\right|^p \dd t\nonumber\\
&\leq & h \left(\int_{t_{k-1}}^{t_k} \left\|\frac{\dd u_{n}}{\dd t}(t)\right\|_{\mathcal A_1}\dd t\right)^p,\quad k=1,\dots,N,
\quad \forall n \in \mathbb{N}.
\end{eqnarray*}
Dividing by $h$, taking the $p$th root and summing over $k=1,\dots,N$ yields that
\begin{equation}\label{le2-mean}
\sum_{k=1}^N \left(\frac{1}{h} \int_{t_{k-1}}^{t_k}\|u_{n}(t) - u_{n}(t'_k)\|^p_{\mathcal{A}_1} \dd t \right)^{\frac{1}{p}}
\leq \int_0^T \left\|\frac{\dd u_{n}}{\dd t}(t)\right\|_{\mathcal A_1}\dd t\qquad \forall n \in \mathbb{N}.
\end{equation}
Let us define
\[ v_k := \left(\frac{1}{h} \int_{t_{k-1}}^{t_k}\|u_{n}(t) - u_{n}(t'_k)\|^p_{\mathcal{A}_1} \dd t \right)^{\frac{1}{p}},
\qquad k=1,\dots, N,\]
and let $v:=(v_1,\dots,v_N)^{\rm T}$ and
\[\|v\|_{\ell_\infty}:=\max_{1 \leq k \leq N} |v_k|.\]
Clearly, for any $p\in[1,\infty)$,
\begin{eqnarray*}
\|v\|_{\ell_p}:= \left(\sum_{k=1}^N |v_k|^p \right)^{\frac{1}{p}} &=& \left(\sum_{k=1}^N |v_k|\,|v_k|^{\left(1-\frac{1}{p}\right)p} \right)^{\frac{1}{p}}\\
&\leq& \|v\|_{\ell_1}^{\frac{1}{p}} \|v\|_{\ell_\infty}^{1-\frac{1}{p}}\leq \|v\|_{\ell_1}^{\frac{1}{p}} \|v\|_{\ell_1}^{1-\frac{1}{p}}
= \|v\|_{\ell_1}.
\end{eqnarray*}
We thus deduce from \eqref{le2-mean}, whose left-hand side is precisely $\|v\|_{\ell_1}$, that
\[
\left(\sum_{k=1}^N \frac{1}{h} \int_{t_{k-1}}^{t_k}\|u_{n}(t) - u_{n}(t'_k)\|^p_{\mathcal{A}_1} \dd t \right)^{\frac{1}{p}}
= \|v\|_{\ell_p}
\leq \int_0^T \left\|\frac{\dd u_{n}}{\dd t}(t)\right\|_{\mathcal A_1}\dd t \qquad \forall n \in \mathbb{N},
\]
and hence, thanks to the second bound in \eqref{le2-eq3} with $p_1=1$,
\begin{eqnarray}\label{le2-mean-2}
\left(\sum_{k=1}^N \|u_{n} - u_{n}(t'_k)\|^p_{L^p(t_{k-1},t_k;\mathcal{A}_1)} \right)^{\frac{1}{p}}
&=& \left(\sum_{k=1}^N \int_{t_{k-1}}^{t_k}\|u_{n}(t) - u_{n}(t'_k)\|^p_{\mathcal{A}_1} \dd t \right)^{\frac{1}{p}}\nonumber\\
&\leq & h^{\frac{1}{p}}\int_0^T \left\|\frac{\dd u_{n}}{\dd t}(t)\right\|_{\mathcal A_1}\dd t \nonumber\\
&\leq & K h^{\frac{1}{p}}\qquad \forall n \in \mathbb{N}.
\end{eqnarray}
It follows from our definition \eqref{le32-defN} of $N=N(\varepsilon)$ that
\begin{equation}\label{le2-hyp}
K h^{\frac{1}{p}} < {\textstyle\frac{1}{3}}\varepsilon.
\end{equation}
For $u \in C([0,T];{\mathcal A}_1)$, we define the piecewise constant interpolant
$$\widehat{u} := \sum_{k=1}^N u(t_k')\,\chi_{(t_{k-1},t_k)},$$
where $\chi_{S}$ is the indicator function for the set $S$.
Thus, by the triangle inequality,
\begin{eqnarray}\label{le2-triangle-p}
&&\hspace{-12mm}
\|u_m - u_n\|_{L^p(0,T; \mathcal{A}_1)}  
\nonumber \\
&&\leq \|u_m - \widehat{u}_m\|_{L^p(0,T; \mathcal{A}_1)} + \|\widehat{u}_m -
\widehat{u}_n\|_{L^p(0,T; \mathcal{A}_1)}
+ \|\widehat{u}_n - u_n\|_{L^p(0,T; \mathcal{A}_1)}
\nonumber\\
&&=
\left(\sum_{k=1}^N \|u_m - u_m(t_k')\|^p_{L^p(t_{k-1},t_k;\mathcal{A}_1)}
\right)^{\frac{1}{p}} + \left(\sum_{k=1}^N \|u_m(t_k') - u_n(t_k')\|^p_{L^p(t_{k-1},t_k;\mathcal{A}_1)}
\right)^{\frac{1}{p}}\nonumber\\
&& \hspace{13mm} + \left(\sum_{k=1}^N \|u_n - u_n(t_k')\|^p_{L^p(t_{k-1},t_k;\mathcal{A}_1)}
\right)^{\frac{1}{p}} \qquad\qquad \forall m, n \in \mathbb{N}.
\end{eqnarray}
Using \eqref{le2-mean-2} and \eqref{le2-hyp} on the first and third term on the right-hand side of \eqref{le2-triangle-p}, we deduce that
\begin{eqnarray}\label{le2-triangle-p2}
\|u_m - u_n\|_{L^p(0,T; \mathcal{A}_1)}
&<& {\textstyle{\frac{2}{3}}}\varepsilon +
\left(\sum_{k=1}^N \|u_m(t_k') - u_n(t_k')\|^p_{L^p(t_{k-1},t_k;\mathcal{A}_1)}
\right)^{\frac{1}{p}}\nonumber\\
&=&  {\textstyle{\frac{2}{3}}}\varepsilon +
\left(\sum_{k=1}^N h \|u_m(t_k') - u_n(t_k')\|^p_{\mathcal{A}_1}
\right)^{\frac{1}{p}}\nonumber\\
&\leq&   {\textstyle{\frac{2}{3}}}\varepsilon +
T^{\frac{1}{p}}\max_{1 \leq k \leq N} \|u_m(t_k') - u_n(t_k')\|_{\mathcal{A}_1}\nonumber\\
&=&  {\textstyle{\frac{2}{3}}}\varepsilon +
T^{\frac{1}{p}}\max_{t' \in G_\varepsilon} \|u_m(t') - u_n(t')\|_{\mathcal{A}_1}    \qquad \forall m, n \in \mathbb{N}.
\end{eqnarray}
As $\{u_n(t')\}_{n\in \mathcal{F}}$ is a Cauchy sequence in $\mathcal{A}_1$ for each $t' \in G_\varepsilon$, and $G_\varepsilon$
is a finite set (of cardinality $N=N(\varepsilon)$), it follows that $\{u_n(t')\}_{n \in \mathcal{F}}$ is a Cauchy sequence in $\mathcal{A}_1$ {\em uniformly} in $t' \in G_\varepsilon$; i.e., for any $\varepsilon>0$ there exists a positive integer $n_0 = n_0(\varepsilon) \in \mathcal{F}$ such that
\[
\max_{t' \in G_\varepsilon} \|u_m(t') - u_n(t')\|_{\mathcal{A}_1} < {\textstyle{\frac{1}{3}}} T^{-\frac{1}{p}} \varepsilon\qquad
\mbox{$\forall m, n \in \mathcal{F}$ such that $m, n \geq n_0$}.\]
Substituting this into \eqref{le2-triangle-p2} we deduce that for any $\varepsilon>0$ there exists a positive integer $n_0
= n_0(\varepsilon) \in \mathcal{F}$ such that
\[ \|u_m - u_n\|_{L^p(0,T; \mathcal{A}_1)} < \varepsilon \qquad \mbox{$\forall m, n \in \mathcal{F}$ such that $m, n \geq n_0$}.\]
Thus we have shown that the sequence $\{u_n\}_{n \in \mathcal{F}}$ is Cauchy (and therefore convergent) in the Banach
space $L^p(0,T; \mathcal{A}_1)$. That completes the proof of part {\textit{(b)}} of the Lemma.
\end{proof}

Now we are ready to prove Theorem \ref{thm:Dubinski}.

\begin{proof}
(The proof of Theorem \ref{thm:Dubinski}). Again, we shall confine ourselves to considering the cases $p \in [1,\infty)$ and $p_1 \in [1,\infty)$. The proof is easily adapted to the cases when $p=\infty$ or $p_1 = \infty$.
The (continuous) embedding ${\mathcal Y}_+ \hookrightarrow L^p(0,T; \mathcal{A}_0)$ is obvious.
It remains to show that the embedding is compact. Let
us consider to this end a sequence $\{u_n\}_{n=1}^\infty \subset {\mathcal Y}_+$ such that
\[ \int_0^T [u_n(t)]_{{\mathcal{M}}_+}^p \dd t \leq K \quad \mbox{and} \quad  \int_0^T \left\|\frac{\dd u_n}{\dd t}(t)\right\|^{p_1}_{\mathcal{A}_1} \dd t \leq K.
\]
Thanks to Lemma \ref{lemma1}, for each $\eta>0$ there exists a constant $c_\eta>0$ such that, for a.e. $t \in [0,T]$ and all
$m, n  \in \mathbb{N}$,
\[ \| u_n(t) - u_m(t) \|_{{\mathcal A}_0} \leq \eta \left( [u_n(t)]_{{\mathcal{M}}_+} + [u_m(t)]_{{\mathcal{M}}_+}\right) + c_\eta \|u_n(t) - u_m(t)\|_{{\mathcal A}_1}.\]
By taking the $p$th power of both sides and integrating over $[0,T]$, we deduce that
\begin{eqnarray}
3^{1-p}\int_0^T \| u_n(t) - u_m(t) \|_{{\mathcal A}_0}^p \dd t   &\leq& \eta^p \int_0^T \left([u_n(t)]_{{\mathcal{M}}_+}^p + [u_m(t)]_{{\mathcal{M}}_+}^p \right) \dd t\nonumber\\
&& + ~c_\eta^p \int_0^T \|u_n(t) - u_m(t)\|_{{\mathcal A}_1}^p\dd t\nonumber\\
& \leq & 2K \eta^p + c_\eta^p \int_0^T \|u_n(t) - u_m(t)\|_{{\mathcal A}_1}^p\dd t
\qquad \forall m, n \in \mathbb{N}.\label{thm1-eq1}
\end{eqnarray}

\smallskip

\textit{(a)} Suppose that $1 \leq p \leq \infty$ and $1<p_1 \leq \infty$.
By hypothesis, $\mathcal{M}_+$ is compactly embedded in $\mathcal{A}_0$, and therefore, since $\mathcal{A}_0$ is (continuously) embedded in $\mathcal{A}_1$, we have that $\mathcal{M}_+$ is compactly embedded in $\mathcal{A}_1$. Hence,
thanks to part \textit{(a)} of
Lemma \ref{lemma2}, the sequence $\{u_n\}_{n=1}^\infty$ is relatively compact in $C([0,T]; \mathcal{A}_1)$, and therefore
also in $L^p(0,T; \mathcal{A}_1)$, for all $p \in [1,\infty]$. Consequently, there exists a subsequence $\{u_{n_k}\}_{k=1}^\infty$, which converges in $L^p(0,T; \mathcal{A}_1)$, for all $p \in [1,\infty]$, and is therefore Cauchy in $L^p(0,T; \mathcal{A}_1)$, for all $p \in [1,\infty]$.

Given $\varepsilon>0$, there exists an $\eta = \eta(\varepsilon)> 0$ such that $3^{p-1} (2K\eta^p) < \frac{1}{2}\varepsilon^p$. For any such $\eta>0$ there exists a positive integer $N_0=N_0(\varepsilon) = N_0(\eta(\varepsilon))$ such that, for all $k, l \geq N_0$,
\[ 3^{p-1} c_\eta^p \int_0^T \|u_{n_k}(t) - u_{n_l}(t)\|_{{\mathcal A}_1}^p\dd t < \textstyle{\frac{1}{2}} \varepsilon^p.\]
Hence, by \eqref{thm1-eq1}, for any $\varepsilon>0$ there exists an $N_0 = N_0(\varepsilon)$ such that, for all $k, l \geq N_0$ we have that
\[ \|u_{n_k} - u_{n_l}\|_{L^p(0,T; \mathcal{A}_0)} = \left(\int_0^T \| u_{n_k}(t) - u_{n_l}(t) \|_{{\mathcal A}_0}^p \dd t\right)^{\frac{1}{p}} < \varepsilon.\]
Consequently, $\{u_{n_k}\}_{k=1}^\infty$ is a Cauchy sequence in $L^p(0,T; \mathcal{A}_0)$. As $\mathcal{A}_0$ is a
Banach space, the normed linear space $L^p(0,T; \mathcal{A}_0)$ is also a Banach space;
therefore $\{u_{n_k}\}_{k=1}^\infty$ is a convergent
sequence in $L^p(0,T; \mathcal{A}_0)$. Thus we have proved that ${\mathcal Y}_+\hookrightarrow\!\!\!\rightarrow L^p(0,T;\mathcal{A}_0)$, and that completes the proof of part \textit{(a)}.

\smallskip

\textit{(b)} Now suppose that $1 \leq p < \infty$ and $p_1=1$. Again,
by hypothesis, $\mathcal{M}_+$ is compactly embedded in $\mathcal{A}_0$, and therefore, since $\mathcal{A}_0$ is (continuously) embedded in $\mathcal{A}_1$, we have that $\mathcal{M}_+$ is compactly embedded in $\mathcal{A}_1$. Hence,
thanks to part \textit{(b)} of
Lemma \ref{lemma2}, the sequence $\{u_n\}_{n=1}^\infty$ is relatively compact in $L^p([0,T]; \mathcal{A}_1)$.
Consequently, there exists a subsequence $\{u_{n_k}\}_{k=1}^\infty$, which converges in $L^p(0,T; \mathcal{A}_1)$
and is therefore Cauchy in $L^p(0,T; \mathcal{A}_1)$. The rest of the argument is identical (verbatim!) to the one
in the second paragraph of the proof of case \textit{(a)} above, and is therefore omitted to avoid unnecessary repetition.
\end{proof}


\section{Discussion}
\label{sec:discussion}

We note that in Dubinski{\u\i} \cite{DUB} the author writes $\mathbb{R}$ instead of our $\mathbb{R}_{\geq 0}$ in \eqref{eq:property} and in property (ii). To be more precise, instead of our \eqref{eq:property} and axioms (i) and (ii) of a seminormed nonnegative cone, Dubinski{\u\i} makes the following assumptions.

Let $\mathcal{A}$ be a linear space over the field $\mathbb{R}$ of real numbers, and suppose that $\mathcal{M}$ is a subset of $\mathcal{A}$ such that
\begin{equation}
\label{eq:property1}
(\forall \varphi \in \mathcal{M})\; (\forall c \in \mathbb{R})\;\;\; c\, \varphi \in \mathcal{M}.
\end{equation}
Dubinski{\u\i} further supposes that each element $\varphi$ of $\mathcal{M}$ in a linear space $\mathcal A$
is assigned a certain real number, denoted by $[\varphi]_{\mathcal M}$, such that:
\begin{enumerate}
\item[(i)$'$] $[\varphi]_{\mathcal M} \geq 0$; and $[\varphi]_{\mathcal M} = 0$ if, and only if, $\varphi=0$; and
\item[(ii)$'$] $(\forall c \in \mathbb{R})\; [c\, \varphi]_{\mathcal M} = |c|\,[\varphi]_{\mathcal M}$,
\end{enumerate}
and calls such a set $\mathcal{M}$ a {\em seminormed set} in $\mathcal{A}$.

The next theorem is akin to Theorem 1
in Dubinski{\u\i} \cite{DUB}, except for our exclusion of the inadmissible case $(p,p_1) = (\infty,1)$, discussed
in Remark \ref{remark2} above.

\begin{theorem}\label{thm:Dubinski2}
Suppose that $\mathcal{A}_0$ and $\mathcal{A}_1$ are Banach spaces, $\mathcal{A}_0 \hookrightarrow \mathcal{A}_1$, and ${\mathcal{M}}$ is a seminormed set in $\mathcal{A}_0$ such that ${\mathcal{M}} \hookrightarrow\!\!\!\rightarrow
\mathcal{A}_0$. Consider the set
\[ {\mathcal Y}:= \left\{\varphi\,:\,[0,T] \rightarrow {\mathcal{M}}\,:\,
[\varphi]_{L^p(0,T;{\mathcal{M}})} + \left\|\frac{{\rm d}\varphi}{{\rm d}t} \right\|_{L^{p_1}(0,T;\mathcal{A}_1)}
< \infty   \right\},
\]
where $1 \leq p \leq \infty$, $1 \leq p_1 \leq \infty$, $\|\cdot\|_{\mathcal{A}_1}$ is the norm of $\mathcal{A}_1$, and ${\rm d}\varphi/{\rm d}t$ is understood in the sense of $\mathcal{A}_1$-valued distributions
on the open interval $(0,T)$. Then, ${\mathcal Y}$, with
\[ [\varphi]_{{\mathcal Y}}:= [\varphi]_{L^p(0,T;{\mathcal{M}})} + \left\|\frac{{\rm d}\varphi}{{\rm d}t} \right\|_{L^{p_1}(0,T;\mathcal{A}_1)},\]
is a seminormed set in the Banach space $L^p(0,T;\mathcal{A}_0)\cap W^{1,p_1}(0,T;\mathcal{A}_1)$. Furthermore,
${\mathcal Y} \!\hookrightarrow\!\!\!\rightarrow \! L^p(0,T; \mathcal{A}_0)$ if: {\em (a)} $1 \leq p \leq \infty$
and $1< p_1 \leq \infty$; or {\em (b)} $1\leq p < \infty$ and $p_1=1$.
\end{theorem}

\begin{proof}
The proof of Theorem \ref{thm:Dubinski2} is identical to that of Theorem \ref{thm:Dubinski}, with $\mathcal{M}_+$, $[\cdot]_{\mathcal{M}_+}$ and $\mathcal{Y}_+$ replaced systematically throughout Lemma \ref{lemma1}, Lemma \ref{lemma2} and the proof of Theorem \ref{thm:Dubinski} with $\mathcal{M}$, $[\cdot]_{\mathcal{M}}$ and $\mathcal{Y}$.
\end{proof}

We shall now provide an alternative proof of Theorem \ref{thm:Dubinski} above by deducing it from Theorem \ref{thm:Dubinski2}.
Suppose to this end that $\mathcal{M}_+$ is a seminormed nonnegative cone in a linear space $\mathcal{A}$, equipped with $[\cdot]_{\mathcal{M}_+}$ satisfying our axioms (i) and (ii). We define
\[ \mathcal{M}_{-}:= -\mathcal{M}_+ = \{ \varphi \in \mathcal{A}\,:\, - \varphi \in \mathcal{M}_{+}\},\]
and we let, for $\varphi \in \mathcal{M}_-$,
\[ [\varphi]_{\mathcal{M}_{-}}:= [-\varphi]_{\mathcal{M}_+}.\]
Further, we define $\mathcal{M}:= \mathcal{M}_+ \cup \mathcal{M}_-$ and we let
\[ [\varphi]_{\mathcal M}:= \left\{
                            \begin{array}{ll}
~[\varphi]_{\mathcal{M}_{+}} & \mbox{when $\varphi \in \mathcal{M}_+$}\\
~[\varphi]_{\mathcal{M}_{-}} & \mbox{when $\varphi \in \mathcal{M}_-$}.
                             \end{array}
                             \right.
\]
Note that neither $\mathcal{M}_+$ nor $\mathcal{M}_-$ is a seminormed set in $\mathcal{A}$ in the sense of Dubinski{\u\i};
however $\mathcal{M}= \mathcal{M}_+ \cup \mathcal{M}_-$ {\em is} a seminormed set in $\mathcal{A}$.

\begin{lemma}
Suppose that $\mathcal{M}_+$ is a seminormed nonnegative cone that is compactly embedded in a normed linear space
$\mathcal{A}$. Then the seminormed set $\mathcal{M}:= \mathcal{M}_+ \cup \mathcal{M}_-$ is compactly embedded in $\mathcal{A}$.
\end{lemma}

\begin{proof}
Suppose that $\{\varphi_n\}_{n=1}^\infty$ is a bounded infinite sequence in $\mathcal{M}$. Then there exists a bounded infinite subsequence of $\{\varphi_n\}_{n=1}^\infty$ all of whose elements belong either to $\mathcal{M}_+$ or
a bounded infinite subsequence of $\{\varphi_n\}_{n=1}^\infty$ all of whose elements
belong to $\mathcal{M}_-$ (the ``or'' here is not an exclusive ``or'', since both
of these situations can occur simultaneously). Whichever the case may be, since both $\mathcal{M}_+$ and $\mathcal{M}_-$ are compactly embedded in $\mathcal{A}$ we can extract a further subsubsequence from the subsequence in question that converges in $\mathcal{A}$. Having extracted a (sub)subsequence from $\{\varphi_n\}_{n=1}^\infty$ that converges in $\mathcal{A}$, we
have shown that $\mathcal{M}$ is compactly embedded in $\mathcal{A}$.
\end{proof}

To summarize, we have thus shown that if $\mathcal{M}_+$ is a seminormed nonnegative cone that is compactly embedded in
$\mathcal{A}_0$, then $\mathcal{M}=\mathcal{M}_+\cup\mathcal{M}_{-}$ is a seminormed set (in the sense of Dubinski{\u\i})
that is compactly embedded in $\mathcal{A}_0$. According to Theorem \ref{thm:Dubinski2},
\[ \mathcal{Y}:= \left\{\varphi\,:\,[0,T] \rightarrow \mathcal{M}\,:\,
[\varphi]_{L^p(0,T;\mathcal M)} + \left\|\frac{{\rm d}\varphi}{{\rm d}t} \right\|_{L^{p_1}(0,T;\mathcal{A}_1)}
< \infty   \right\}
\]
is a seminormed set equipped with
\[ [\varphi]_{\mathcal Y}:= [\varphi]_{L^p(0,T;\mathcal M)} + \left\|\frac{{\rm d}\varphi}{{\rm d}t} \right\|_{L^{p_1}(0,T;\mathcal{A}_1)},\]
and $\mathcal{Y} \hookrightarrow\!\!\!\rightarrow L^p(0,T; \mathcal{A}_0)$ for the ranges of $p$ and $p_1$ as in
Theorem \ref{thm:Dubinski2}.

We shall now use this observation to prove that the seminormed nonnegative cone
\[ \mathcal{Y}_+:= \left\{\varphi\,:\,[0,T] \rightarrow \mathcal{M}_+\,:\,
[\varphi]_{L^p(0,T;\mathcal{M}_+)} + \left\|\frac{{\rm d}\varphi}{{\rm d}t} \right\|_{L^{p_1}(0,T;\mathcal{A}_1)}
< \infty   \right\}
\]
is compactly embedded in $L^p(0,T; \mathcal{A}_0)$, which is precisely the result stated in Theorem \ref{thm:Dubinski}.
Suppose to this end that $\{\varphi_n\}_{n=1}^\infty \subset \mathcal{Y}_+$ and that there exists
a $K > 0$ such that
\[
[\varphi_n]_{L^p(0,T;\mathcal{M}_+)} + \left\|\frac{{\rm d}\varphi_n}{{\rm d}t} \right\|_{L^{p_1}(0,T;\mathcal{A}_1)}
\leq K,
\]
for all $n=1,2,\dots$. Since $\mathcal{M}_+ \subset \mathcal{M}$, we have that $\mathcal{Y}_+ \subset \mathcal{Y}$, and therefore
$\{\varphi_n\}_{n=1}^\infty$ is a bounded sequence in $\mathcal{Y}$; since $\mathcal{Y}
\hookrightarrow\!\!\!\rightarrow L^p(0,T; \mathcal{A}_0)$, it follows that $\{\varphi_n\}_{n=1}^\infty$ has a convergent subsequence in $L^p(0,T; \mathcal{A}_0)$. Thus we have shown that $\mathcal{Y}_+ \hookrightarrow\!\!\!\rightarrow L^p(0,T; \mathcal{A}_0)$, as was claimed in Theorem \ref{thm:Dubinski} above for the stated ranges of $p$ and $p_1$.


\section{Connections with Maitre's nonlinear compactness lemma}
\label{sec:maitre}

Our main theorem, Theorem \ref{thm:Dubinski}, was proved using two lemmas, Lemma \ref{lemma1} and Lemma \ref{lemma2}. Here we shall show that in certain instances an alternative proof of Lemma \ref{lemma2} (and thereby, indirectly,
also of Theorem \ref{thm:Dubinski}) can be given by using a nonlinear
compactness result due to Maitre \cite{maitre}. An analogous argument can be devised in the
case of Theorem \ref{thm:Dubinski2}.

Suppose that $\mathcal{A}$ and $\mathcal{A}_1$ are two Banach spaces, $T>0$ as before, and
$1 \leq p \leq \infty$. Throughout this section $B$ will denote a (nonlinear) {\em compact} operator from
$\mathcal{A}$ to $\mathcal{A}_1$, i.e. one that maps bounded subsets of $\mathcal{A}$ into
relatively compact subsets of $\mathcal{A}_1$. The next theorem is due to Maitre \cite{maitre}.

\begin{theorem}\label{thm:maitre}
Let $\mathcal{U}$ be bounded subset of $L^1(0,T; \mathcal{A})$ such that $\mathcal{V}=B(\mathcal{U})$
is a subset of $L^p(0,T; \mathcal{A}_1)$ and is bounded in $L^r(0,T;\mathcal{A}_1)$ with $r>1$.
Assume further that
\[ \lim_{h \rightarrow 0_+} \|v(\cdot + h) - v\|_{L^p(0,T-h; \mathcal{A}_1)} = 0,
\qquad \mbox{uniformly for $v \in \mathcal{V}$}.\]
Then, $\mathcal{V}$ is relatively compact in $L^p(0,T;\mathcal{A}_1)$ (and in $C([0,T]; \mathcal{A}_1)$
if $p=\infty$).
\end{theorem}

The following result is a slight generalization of Corollary 2.4 in Maitre \cite{maitre},
and is a direct consequence of Theorem \ref{thm:maitre} above, using Lemma 4 on p.~77 in Simon \cite{Simon}.

\begin{corollary}\label{corr}
Let $\mathcal{U}$ be a bounded subset of $L^1(0,T;\mathcal{A})$ such that $\mathcal{V}=B(\mathcal{U})$
is bounded in $L^r(0,T;\mathcal{A}_1)$ with $r>1$. Assume that
\[ \frac{\dd \mathcal{V}}{\dd t} := \left\{ \frac{\dd v}{\dd t}\,:\, v \in \mathcal{V}\right\}\]
is bounded in $L^1(0,T;\mathcal{A}_1)$ (respectively, in $L^{p_1}(0,T;\mathcal{A}_1)$ with $p_1>1$).
Then, $\mathcal{V}$ is relatively compact in $L^p(0,T;\mathcal{A}_1)$ for any $p < \infty$
(respectively, in $C([0,T]; \mathcal{A}_1))$.
\end{corollary}
Suppose that $\mathcal{A}$ and $\mathcal{A}_1$ are Banach spaces and $B\,:\, \mathcal{A} \rightarrow \mathcal{A}_1$ is an injective nonlinear operator on $\mathcal{A}$ that is
positively homogeneous of degree $\alpha \geq 1$ (i.e., $B(c u) = c^{\alpha} Bu$ for all $c \in \mathbb{R}_{\geq 0}$ and all $u \in \mathcal{A}$). It then follows that the range of $B$,
which we shall denote by $\mathcal{M}_+$, is a nonnegative cone in $\mathcal{A}_1$. When
equipped with
\[ [v]_{\mathcal{M}_+} :=\|B^{-1}v\|_{\mathcal{A}}^\alpha,\quad v \in \mathcal{M}_+,\]
$\mathcal{M}_+$ is a seminormed nonnegative cone in the ambient Banach space $\mathcal{A}_1$. We shall assume that ${\mathcal{M}}_+ \hookrightarrow\!\!\!\rightarrow \mathcal{A}_1$.
Consider the nonnegative seminormed cone $\mathcal{Y}_+$,
defined in Lemma \ref{lemma2}, with: \textit{(a)} $1 \leq p \leq \infty$ and $1 < p_1 \leq \infty$;
or with \textit{(b)} $1 \leq p < \infty$ and $p_1=1$.
Note in particular that, either way, $\mathcal{Y}_+ \subset L^1(0,T; \mathcal{A}_1)$, and
therefore $\mathcal{Y}_+ \subset W^{1,1}(0,T; \mathcal{A}_1)$.
We shall show that in the special case when $\mathcal{M}_+$, $[~\!\cdot~\!]_{{\mathcal M}_+}$ and $\mathcal{Y}_+$ are thus defined, Corollary \ref{corr} implies Lemma \ref{lemma2}.

Let us suppose to this end that $\mathcal{V}$ is a
bounded subset of $\mathcal{Y}_+$. It then follows that $\mathcal{V}$ is a bounded subset of $W^{1,1}(0,T; {\mathcal A}_1)$; hence,
by H\"older's inequality on $(0,T)$, $\mathcal{U}=B^{-1}\mathcal{V}$ is a bounded subset of $L^1(0,T; \mathcal{A})$.
On noting that $W^{1,1}(0,T)$ is continuously embedded in $C[0,T]$, and therefore also in $L^r(0,T)$ for
all $r \geq 1$, it follows that $\mathcal{V}$ is a bounded subset of $L^r(0,T; \mathcal{A}_1)$ with $r \geq 1$
(and in particular with $r>1$). Furthermore, $\dd\mathcal{V}/\dd t$ is a bounded subset of $L^{p_1}(0,T;\mathcal{A}_1)$
with $p_1>1$ in case \textit{(a)} and $p_1=1$ in case \textit{(b)}. Thanks to our assumption
that ${\mathcal{M}}_+ \hookrightarrow\!\!\!\rightarrow \mathcal{A}_1$, the operator
$B\,:\, \mathcal{A} \rightarrow \mathcal{A}_1$ is compact.
Thus, by Corollary \ref{corr}, $\mathcal{V}$ is
relatively compact in $C([0,T]; \mathcal{A}_1)$ in case \textit{(a)}, and in $L^p(0,T;\mathcal{A}_1)$
in case \textit{(b)}. Consequently, ${\mathcal{Y}}_+ \hookrightarrow\!\!\!\rightarrow C([0,T];\mathcal{A}_1)$
in case \textit{(a)}, and ${\mathcal{Y}}_+ \hookrightarrow\!\!\!\rightarrow L^p(0,T;\mathcal{A}_1)$ in case
\textit{(b)}; and hence we arrive at the conclusions of Lemma \ref{lemma2}, with $\mathcal{M}_+$, $[~\!\cdot~\!]_{{\mathcal M}_+}$ and $\mathcal{Y}_+$ as defined in the previous paragraph. Lemmas \ref{lemma1} and \ref{lemma2} then imply Theorem \ref{thm:Dubinski} with such $\mathcal{M}_+$, $[~\!\cdot~\!]_{{\mathcal M}_+}$ and $\mathcal{Y}_+$.

\bigskip

\noindent
\textbf{Acknowledgement:} We are grateful to Xiuqing Chen (Beijing) for helpful comments.
\bibliographystyle{siam}

\bibliography{polyjwbesrefs}

\end{document}